\documentclass[final]{siamonline171218}

\usepackage{graphicx}
\usepackage[caption=false]{subfig}
\usepackage{amsmath,amsfonts,amsopn,multirow}
\ifpdf
  \DeclareGraphicsExtensions{.eps,.pdf,.png,.jpg}
\else
  \DeclareGraphicsExtensions{.eps}
\fi
\usepackage{multirow}

\numberwithin{theorem}{section}

\newsiamremark{remark}{Remark}

\newcommand{\TheTitle}{Fast Nonoverlapping Block Jacobi Method for the Dual Rudin--Osher--Fatemi Model} 
\newcommand{\TheAuthors}{C.-O. Lee and J. Park}

\headers{Fast Block Jacobi for Rudin--Osher--Fatemi}{\TheAuthors}

\title{{\TheTitle}\thanks{
\funding{This work was supported by NRF grant funded by MSIT (NRF-2017R1A2B4011627).}}}

\author{
  Chang-Ock Lee\thanks{Department of Mathematical Sciences, KAIST, Daejeon 34141, Korea \newline
    (\email{colee@kaist.edu}, \email{jongho.park@kaist.ac.kr}).}
  \and
  Jongho Park\footnotemark[2]
}

\ifpdf
\hypersetup{
  pdftitle={\TheTitle},
  pdfauthor={\TheAuthors}
}
\fi

\newcommand\p{\mathbf{p}}
\newcommand\q{\mathbf{q}}
\renewcommand\r{\mathbf{r}}

\newcommand\tOmega{\tilde{\Omega}}

\newcommand\N{\mathcal{N}}

\renewcommand\S{\mathcal{S}}
\newcommand\bS{\bar{\mathcal{S}}}
\renewcommand\P{\mathcal{P}}

\newcommand\T{\mathcal{T}}

\renewcommand\div{\mathrm{div}}

\newcommand\supp{\mathrm{supp}\,}

\DeclareMathOperator*{\argmin}{\arg\min}

\usepackage{algorithm, algorithmic}

\begin{document}

\maketitle

% Abstract
\begin{abstract}
We consider nonoverlapping domain decomposition methods for the Rudin--Osher--Fatemi~(ROF) model, which is one of the standard models in mathematical image processing.
The image domain is partitioned into rectangular subdomains and local problems in subdomains are solved in parallel.
Local problems can adopt existing state-of-the-art solvers for the ROF model.
We show that the nonoverlapping relaxed block Jacobi method for a dual formulation of the ROF model has the $O(1/n)$ convergence rate of the energy functional, where~$n$ is the number of iterations.
Moreover, by exploiting the forward-backward splitting structure of the method, we propose an accelerated version whose convergence rate is~$O(1/n^2)$.
The proposed method converges faster than existing domain decomposition methods both theoretically and practically, while the main computational cost of each iteration remains the same.
We also provide the dependence of the convergence rates of the block Jacobi methods on the image size and the number of subdomains.
Numerical results for comparisons with existing methods are presented.
\end{abstract}

% Keywords

% AMS classification

% Section: Introduction
\section{Introduction}
\label{Sec:Introduction}

As large-scale images have become available in these days, it takes a long time to process such images and demands to reduce the time, especially in real time, are required.
To accomplish such goal, there has been arisen the necessity of mathematical methods to make efficient use of distributed memory computers.
In this perspective, we consider domain decomposition methods~(DDMs), which have been already used successfully for elliptic partial differential equations~\cite{QV:1999,TW:2005}.

In DDMs, a large-scale problem defined on a domain $\Omega$ is splitted into smaller ones called local problems defined on subdomains $\left\{ \Omega_s \right\}_{s=1}^{\N}$ satisfying $\bar{\Omega} = \bigcup_{s=1}^{\N} \bar{\Omega}_s$.
We assign one local problem per processor so that all local problems can be solved in parallel in distributed memory computers.
A solution of the full-dimension problem is obtained by the assembly of solutions of local problems.
Thus, DDMs are suitable for distributed memory computers.
Block Jacobi and Gauss--Seidel methods are typical examples of DDMs.

% The ROF model
In this paper, we consider block Jacobi methods for the Rudin--Osher--Fatemi~(ROF) model~\cite{ROF:1992}, which is one of the standard models in variational image processing given by
\begin{equation}
\label{ROF}
\min_{u \in BV(\Omega)} \left\{ \frac{\alpha}{2} \int_{\Omega} (u-f)^2 \,dx + TV(u) \right\},
\end{equation}
where $\alpha$ is a positive parameter, $f$ is a corrupted image, $TV(u)$ is the total variation of $u$, and $BV(\Omega)$ is the space of functions of bounded variation on an image domain~$\Omega \subset \mathbb{R}^2$.
There have been numerous approaches to solve~\cref{ROF}; see~\cite{BT:2009,Chambolle:2004,CP:2011,GO:2009,WYYZ:2008,WT:2010} for instance.

% DDMs for the ROF model
There have been several remarkable preceding researches on DDMs for variational image processing.
Subspace correction methods for~\cref{ROF} were proposed in~\cite{FLS:2010,FS:2009}, and various techniques to deal with local problems were considered in~\cite{LOS:2013,LLWY:2016}.
In~\cite{HL:2013}, subspace correction methods for the case of mixed $L^1$/$L^2$-fidelity were considered.
However, it was shown in~\cite{LN:2017} that such subspace correction methods for~\cref{ROF} may not converge to a minimizer.
To ensure convergence to a minimizer, in~\cite{CTWY:2015,HL:2015,LN:2017}, subspace correction methods for the Fenchel--Rockafellar dual formulation of~\cref{ROF} given by
\begin{equation}
\label{dual_ROF}
\min_{\p \in C_0^1 (\Omega ; \mathbb{R}^2)} \frac{1}{2} \int_{\Omega} (\div \p + \alpha f )^2 \,dx \hspace{0.5cm}
\textrm{subject to } |\p(x)| \leq 1 \hspace{0.2cm}\forall x \in \Omega
\end{equation}
were proposed, where $C_0^1 (\Omega ; \mathbb{R}^2)$ is the set of continuously differentiable vector fields $\p$:~$\Omega \rightarrow \mathbb{R}^2$ with compact support.
Recently, interface-based domain decomposition methods for total variation minimization were proposed~\cite{LNP:2019,LPP:2019}.

% Coordinate descent = Gauss--Seidel
Over the past decade, block coordinate descent type algorithms for nonsmooth convex optimization have been successfully developed~\cite{BT:2013,Tseng:2001}.
They solve a minimization problem by successively minimizing the functional with respect to one of the coordinate blocks.
In the context of domain decomposition methods, they correspond to nonoverlapping block Gauss--Seidel methods.
Indeed, we observe that convergence rate analyses for block coordinate descent methods given in~\cite{CP:2015,ST:2016} can be directly applied to the nonoverlapping block Gauss--Seidel method for~\cref{dual_ROF} in~\cite{HL:2015}.
However, to the best of our knowledge, there are no available analysis of the convergence rate for the nonoverlapping block Jacobi method for~\cref{dual_ROF}.

% In this paper..
In this paper, we prove the $O(1/n)$ convergence rate of the nonoverlapping relaxed block Jacobi method for~\cref{dual_ROF}, which was proposed in~\cite{HL:2015}.
We note that it was shown in~\cite{CTWY:2015} that the overlapping counterpart is $O(1/n)$ convergent and the authors gave a remark that their proof cannot be directly applied to the nonoverlapping case.
In addition to the analysis of the convergence rate, we propose an accelerated nonoverlapping block Jacobi method for~\cref{dual_ROF} which has the $O(1/n^2)$ convergence rate.
First, we modify the relaxed block Jacobi method so that it has the forward-backward splitting structure~\cite{CW:2005}.
We call such a modified method as \textit{pre-relaxed} block Jacobi method.
Thanks to its forward-backward structure, we obtain an $O(1/n^2)$ convergent algorithm by adding the momentum technique introduced in~\cite{BT:2009}.
Numerical results ensure faster convergence of the proposed method compared to the existing ones.

% Paper organization
The rest of the paper is organized as follows.
Basic settings for the paper are presented in \cref{Sec:Preliminaries}.
We analyze the convergence rate of the relaxed block Jacobi method and propose its variants including an~$O(1/n^2)$ convergent method in \cref{Sec:Jacobi}.
We consider how to deal with local problems of the proposed methods in \cref{Sec:Local}.
We compare the proposed methods to the existing ones by numerical experiments in \cref{Sec:Numerical}.
We conclude the paper with remarks in \cref{Sec:Conclusion}.

% Section: Preliminaries
\section{Preliminaries}
\label{Sec:Preliminaries}

% Table: Notations
\begin{table}
\centering
\small
\begin{tabular}{| c | c | c |} \hline
notation & description & section \\ \hline 
$\Omega$ 	& image domain of size $M \times N$ & 2.1 \\
$V$ 		& a collection of functions: $\Omega \rightarrow \mathbb{R}$ & 2.1 \\
$W$			& a collection of functions: $\Omega \rightarrow \mathbb{R}^2$ & 2.1 \\
$C$			& $\{ \p \in W : |\p_{ij}| \leq 1 \quad \forall (i,j) \in W \}$ & 2.1 \\
$\nabla$	& discrete gradient operator & 2.1 \\
$\div$		& discrete divergence operator, $\div = - \nabla^*$ & 2.1 \\
$F$	& $F(\p) = \frac{1}{2} \| \div \p + \alpha f \|_2^2$, dual energy functional & 2.1 \\
$\chi_C$ & characteristic function of $C \subset W$ & 2.1 \\
$\N$		& $M_s \times N_s$, number of subdomains in the decomposition $\{ \Omega_s \}_{s=1}^{\N}$ & 2.2 \\
$N_c$		& number of colors on the decomposition $\{ \Omega_s \}_{s=1}^{\N}$ & 2.2 \\
$S_k$		& union of all subdomains with color $k$ & 2.2 \\
$W_k$		& a collection of functions: $S_k \rightarrow \mathbb{R}^2$ & 2.2 \\
$R_k$		& $R_k \p = \p|_{S_k}$, restriction operator: $W \rightarrow W_k$ & 2.2 \\
$C_k$		& $R_k C$ & 2.2 \\
$F_k$ & $F_k (\p_k ; \r) = F(R_k^* \p_k + (I-R_k^* R_k) \r )$, local energy functional & 2.2 \\
$D_k$ & Bregman distance associated with $F_k (\cdot ; \r)$ & 2.2 \\
$D$ 		& $D(\p, \q) = \sum_{k=1}^{N_c} D_k (R_k \p, R_k \q)$ & 2.2 \\
$\S_k$ 	& $\S_k (\p) = \argmin_{\p_k \in W_k} \{ F_k (\p_k ; \p) + \chi_{C_k} (\p_k) \}$ & 2.2 \\
$\bS_k$ 	& $\bS_k (\p) = R_k^* \S_k (\p) + (I-R_k^* R_k) \p$ & 2.2 \\
$\P_k$		& $\P_k (\p) = \argmin_{\p_k \in W_k} \{ F_k (N_c \p_k - (N_c - 1)R_k \p ; \p) + \chi_{C_k} (\p_k) \}$ & 3.2 \\
\hline
\end{tabular}
\caption{List of notations used in \cref{Sec:Preliminaries,Sec:Jacobi}}
\label{Table:notations}
\end{table}

In this section, we review the standard finite difference framework for the ROF model.
Then, we formulate the Fenchel--Rockafellar dual problem of the discrete ROF model.
Finally, a nonoverlapping domain decomposition setting for the dual problem is presented.
A list of notations used in \cref{Sec:Preliminaries,Sec:Jacobi} is given in \cref{Table:notations}.

% Subsection: Discrete setting
\subsection{Discrete setting}
We assume that the grayscale image domain~$\Omega$ is composed of $M \times N$ pixels; i.e.,
\begin{equation*}
\Omega = \left\{ (i,j) : 1\leq i \leq M, \hspace{0.1cm} 1 \leq j \leq N \right\}.
\end{equation*}

Let $V$ and $W$ be the collections of functions from~$\Omega$ into~$\mathbb{R}$ and $\mathbb{R}^2$, respectively.
Both spaces are equipped with the usual Euclidean inner products $\left< \cdot , \cdot \right>$ and their induced norms $\| \cdot \|_2$.
In addition, let $C$ be the convex subset of $W$ given by
\begin{equation*}
C = \left\{ \p \in W : |\p_{ij}| \leq 1 \hspace{0.2cm}\forall (i,j) \in \Omega \right\},
\end{equation*}
where $|\p_{ij}| = (|p_{ij}^1|^2 + |p_{ij}^2|^2 )^{1/2}$ for $\p = (p^1 , p^2)$.

The discrete gradient operator $\nabla$:~$V \rightarrow W$ is defined by
\begin{eqnarray*}
(\nabla u)^1_{ij} &=& \left\{\begin{array}{cl} u_{i+1,j} - u_{ij}  & \textrm{if} \quad i=1,\dots,M-1,   \\     0   & \textrm{if} \quad i=M, \end{array} \right. \\
(\nabla u)^2_{ij} &=& \left\{\begin{array}{cl} u_{i,j+1} - u_{ij}   & \textrm{if} \quad j=1,\dots,N-1, \\     0 & \textrm{if} \quad j=N. \end{array} \right.
\end{eqnarray*}
We define the discrete divergence operator $\div$:~$W \rightarrow V$ as the minus adjoint of~$\nabla$, i.e.,
\begin{eqnarray*}
(\div \p)_{ij} &=& \begin{cases} p_{ij}^1 & \textrm{if} \quad i=1, \\ p_{ij}^1 - p_{i-1,j}^1 & \textrm{if} \quad i=2,\dots,M-1, \\ -p_{i-1,j}^1 & \textrm{if} \quad i=M \end{cases}\\
&& \hspace{0.1cm} + \begin{cases} p_{ij}^2 & \textrm{if} \quad j=1, \\ p_{ij}^2 - p_{i,j-1}^2 & \textrm{if} \quad j=2,\dots,N-1, \\ -p_{i,j-1}^2 & \textrm{if} \quad j=N.\end{cases}
\end{eqnarray*}

With the operators defined above, a discrete version of~\cref{ROF} is stated as
\begin{equation}
\label{d_ROF}
\min_{u \in V} \left\{ \frac{\alpha}{2} \| u - f \|_2^2 + \| \nabla u \|_1 \right\},
\end{equation}
where $\| \nabla u \|_1$ is the discrete total variation of $u$ defined as the 1-norm of $\nabla u$:
\begin{equation*}
\| \nabla u \|_1 = \sum_{(i,j) \in \Omega} |(\nabla u)_{ij}|.
\end{equation*}

Now, we consider the Fenchel--Rockafellar dual problem of~\cref{d_ROF}.
We define the energy functional $F$:~$W \rightarrow \mathbb{R}$ by
\begin{equation*}
F(\p) = \frac{1}{2} \| \div \p + \alpha f \|_2^2 , \hspace{0.5cm} \p \in W.
\end{equation*}
It is well-known that a solution~$u^*$ of~\cref{d_ROF} can be recovered from a solution~$\p^*$ of
\begin{equation}
\label{d_dual_ROF}
\min_{\p \in W} \left\{ F(\p) + \chi_C (\p) \right\}
\end{equation}
by a simple algebraic formula
\begin{equation*}
u^* = f + \frac{1}{\alpha}\div \p^* .
\end{equation*}
Here $\chi_C$:~$W \rightarrow \bar{\mathbb{R}}$ is the characteristic function of $C \subset W$ defined as
\begin{equation*}
\chi_C (\p) = \begin{cases} 0 & \textrm{if} \quad \p \in C, \\ \infty & \textrm{if} \quad \p \not\in C.\end{cases}
\end{equation*}
We call~\cref{d_dual_ROF} the Fenchel--Rockafellar dual problem of~\cref{d_ROF}.
For more details, readers may refer~\cite{Chambolle:2004,Rockafellar:2015}.

% Subsection: Domain decomposition setting
\subsection{Domain decomposition setting}
First, we partition the image domain $\Omega$ into $\N = M_s \times N_s$ nonoverlapping rectangular subdomains $\left\{ \Omega_s \right\}_{s=1}^{\N}$.
All subdomains can be classified into~$N_c$ classes~(colors) by the usual coloring technique~(see section~2.5.1 of~\cite{TW:2005}); local problems on subdomains in the same class are solved independently.
Details of the coloring technique will be given in \cref{Sec:Numerical}.
Let $S_k$ be the union of all subdomains with color~$k$ for $k=1,\dots,N_c$.
Then, we have
\begin{equation*}
\Omega = \bigcup_{k=1}^{N_c} S_k, \hspace{0.5cm} S_k \cap S_j = \emptyset \quad \forall k,j.
\end{equation*}

For each $k=1,\dots,N_c$, we define the local function space $W_k$ as the collection of functions from~$S_k$ to~$\mathbb{R}^2$.
Also, we define the restriction operator $R_k$:~$W \rightarrow W_k$ as
\begin{equation}
\label{R_k}
R_k \p = \p |_{S_k}, \hspace{0.5cm} \p \in W.
\end{equation}
Then, its adjoint $R_k^*$:~$W_k \rightarrow W$ becomes the natural extension operator:
\begin{equation*}
(R_k^* \p_k )_{ij} = \begin{cases} (\p_k)_{ij} & \textrm{if} \quad (i,j) \in S_k , \\ \mathbf{0} & \textrm{if} \quad (i,j) \not\in S_k , \end{cases}
\hspace{0.5cm} \p_k \in W_k.
\end{equation*}
Note that $R_k^* R_k$ is the orthogonal projection from $W$ onto $W_k$ and
\begin{equation*}
\sum_{k=1}^{N_c} R_k^* R_k = I.
\end{equation*}
We clearly have $W = \bigoplus_{k=1}^{N_c} R_k^* W_k$.
Similarly, by defining $C_k = R_k C$, we have $C = \bigoplus_{k=1}^{N_c} R_k^* C_k$ and
\begin{equation*}
\chi_C (\p) =\sum_{k=1}^{N_c} \chi_{C_k} (R_k \p), \hspace{0.5cm} \p \in W.
\end{equation*}

Given $\r \in W$, the local energy functional $F_k$:~$W_k \rightarrow \mathbb{R}$ is defined as
\begin{equation*}
F_k (\p_k ; \r) = F(R_k^* \p_k + (I-R_k^* R_k) \r), \hspace{0.5cm} \p_k \in W_k.
\end{equation*}
The Bregman distance~\cite{Bregman:1967} associated with $F_k (\cdot ; \r)$ is denoted by~$D_k$, i.e.,
\begin{equation}
\label{Bregman1}
D_k (\p_k , \q_k) = F_k (\p_k ; \r) - F_k (\q_k ; \r) -  \left< \nabla F_k (\q_k ; \r) , \p_k - \q_k \right>, \hspace{0.5cm} \p_k, \q_k \in W_k,
\end{equation}
and one can observe that it is independent of $\r$.
Indeed, we have
\begin{equation}
\label{Bregman2}
D_k (\p_k , \q_k) =\frac{1}{2} \| \div R_k^* (\p_k - \q_k) \|_2^2 .
\end{equation}
We define the function $D$:~$W \times W \rightarrow \mathbb{R}$ as
\begin{equation}
\label{D_dist}
D (\p , \q) = \sum_{k=1}^{N_c} D_k (R_k \p , R_k \q), \hspace{0.5cm} \p, \q \in W.
\end{equation}
We provide the descent lemma corresponding to~$D$.

% Lemma: Lipschitz properties of D_k and D
\begin{lemma}
\label{Lem:Lip}
For any $\p, \q \in W$, we have
\begin{equation*}
F(\p) \leq F(\q) + \left< \nabla F(\q), \p - \q \right> + N_c D (\p, \q).
\end{equation*}
\end{lemma}
\begin{proof}
By direct calculation, we have
\begin{align*}
F(\p) - F(\q) &- \left< \nabla F(\q), \p - \q \right> = \frac{1}{2} \| \div (\p - \q) \|_2^2 \\
&= \frac{1}{2} \left\| \div \left( \sum_{k=1}^{N_c}R_k^* R_k (\p - \q) \right) \right\|_2^2 \\
&\leq N_c \sum_{k=1}^{N_c}  \frac{1}{2} \| \div R_k^* R_k (\p - \q) \|_2^2 = N_c D(\p, \q).
\end{align*}
\end{proof}

Finally, we define two operators $\S_k$:~$W \rightarrow W_k$ and $\bS_k$:~$W \rightarrow W$ by
\begin{equation}
\label{S_k}
\S_k (\p) = \argmin_{\p_k \in W_k} \left\{ F_k (\p_k ; \p) + \chi_{C_k} (\p_k) \right\} , \hspace{0.5cm} \p \in W,
\end{equation}
and
\begin{equation*}
\bS_k (\p) = R_k^* \S_k (\p) + (I - R_k^* R_k ) \p, \hspace{0.5cm} \p \in W.
\end{equation*}
As the local energy functional $F_k$ is not strongly convex, the minimization problem in~\cref{S_k} may admit nonunique minimizers.
In this case, we take $\S_k (\p)$ as \textit{any} one among them.
By the optimality of $\S_k$, we have the following lemma.

% Lemma: Optimality of S_k
\begin{lemma}
\label{Lem:opt}
For any $\p_k \in C_k$ and $\q \in C$, we have
\begin{equation*}
F_k (\p_k; \q) - F(\bS_k (\q)) \geq  \frac{1}{2} \| \div R_k^* (\p_k - R_k \bS_k (\q)) \|_2^2 .
\end{equation*}
\end{lemma}
\begin{proof}
It is clear from the optimality condition of $\S_k (\q)$~(see the proof of Lemma~3.3 in~\cite{CTWY:2015}).
\end{proof}

\Cref{Lem:opt} shall be useful in the convergence analysis of block Jacobi methods in the next section.
We note that it was also used in the case of overlapping domain decomposition~\cite{CTWY:2015}.

% Section: Block Jacobi methods
\section{Block Jacobi methods}
\label{Sec:Jacobi}
In this section, we analyze the convergence rate of the relaxed block Jacobi method~(Algorithm~2 in~\cite{HL:2015}) for~\cref{d_dual_ROF}.
Then, by observing a resemblance between the relaxed block Jacobi method and the forward-backward splitting~\cite{BT:2009,CW:2005}, we propose a variant of the method called the pre-relaxed block Jacobi method, which has the exact forward-backward splitting structure.
Based on the pre-relaxed block Jacobi method, we also propose an accelerated block Jacobi method which has the~$O(1/n^2)$ convergence rate.

% Subsection: Relaxed Block Jacobi method
\subsection{Relaxed Block Jacobi method}
The relaxed block Jacobi method for~\cref{d_dual_ROF} is presented in the following, which was first proposed in~\cite{HL:2015}.

% Algorithm (Relaxed Block Jacobi Method)
\begin{algorithm}[]
\caption{Relaxed Block Jacobi Method}
\begin{algorithmic}[]
\label{Alg:RJ}
\STATE Let $\p^{(0)} \in C$.
\FOR {$n=0,1,2,\dots$}
\STATE $\displaystyle \p^{(n+1)} = \frac{1}{N_c} \sum_{k=1}^{N_c} \bS_k (\p^{(n)})$
\ENDFOR
\end{algorithmic}
\end{algorithm}

Note that
\begin{equation*}
\p^{(n+1)} = \frac{1}{N_c} \sum_{k=1}^{N_c} \bS_k (\p^{(n)}) = \left( 1 - \frac{1}{N_c} \right) \p^{(n)} + \frac{1}{N_c} \sum_{k=1}^{N_c} R_k^* \S_k (\p^{(n)})
\end{equation*}
in \cref{Alg:RJ}.
That is,~$\p^{(n+1)}$ is obtained by relaxation of~$\sum_{k=1}^{N_c} R_k^* \S_k (\p^{(n)})$ with~$\p^{(n)}$, and it is the reason why we call \cref{Alg:RJ} the relaxed block Jacobi method.

In~\cite{HL:2015}, the convergence of \cref{Alg:RJ} was shown without the convergence rate.
It is summarized in the following proposition.

% Proposition: Convergence analysis of Relaxed Block Jacobi Method by Hinterm\"{u}ller and Langer
\begin{proposition}
\label{Prop:HL_RJ}
Let $\left\{ \p^{(n)} \right\}$ be the sequence generated by \cref{Alg:RJ}.
Then, it satisfies the followings: \\
\emph{(i)} The sequence $\left\{ F(\p^{(n)}) \right\}$ is decreasing, so that it converges. \\
\emph{(ii)} The sequence $\left\{ \p^{(n)} \right\}$ is bounded and has a limit point which is a solution of~\cref{d_dual_ROF}.
\end{proposition}

First, we prove that the convergence rate of \cref{Alg:RJ} is~$O(1/n)$.
The following lemma is a main ingredient for our proof.

% Lemma: Quasi-Forward-Backward structure of the relaxed Jacobi method
\begin{lemma}
\label{Lem:RJ_FB}
For any $\p, \q \in C$, we have
\begin{equation*}
\frac{1}{N_c}F(\p) + \left( 1 - \frac{1}{N_c} \right) F(\q) - F\left( \frac{1}{N_c}\sum_{k=1}^{N_c} \bS_k (\q) \right)
\geq D\left( \p , \frac{1}{N_c}\sum_{k=1}^{N_c} \bS_k (\q) \right) - D(\p , \q).
\end{equation*}
\end{lemma}
\begin{proof}
For the sake of convenience, let 
\begin{equation*}
\S (\q) = \frac{1}{N_c} \sum_{k=1}^{N_c} \bS_k (\q).
\end{equation*}
By the convexity of $F$, we have
\begin{equation}
\label{RJ_FB1}
\frac{1}{N_c}\sum_{k=1}^{N_c} F(\bS_k (\q)) \geq F (\S (\q)).
\end{equation}
Invoking \cref{Lem:opt} with $\p_k = R_k \p$ and summing over $k=1, \dots, N_c$ give
\begin{equation}
\label{RJ_FB2}
\frac{1}{N_c}\sum_{k=1}^{N_c} \left[ F_k (R_k \p; \q) - F(\bS_k (\q)) \right] \geq  \frac{1}{N_c} \sum_{k=1}^{N_c}\frac{1}{2} \| \div R_k^* R_k (\p - \bS_k (\q)) \|_2^2 .
\end{equation}
Also, by the relation
\begin{equation*}
R_k^* R_k \left( \p - \S (\q) \right) = \left( 1 - \frac{1}{N_c} \right) R_k^* R_k (\p - \q) + \frac{1}{N_c}R_k^* R_k (\p -\bS_k (\q ))
\end{equation*}
and the convexity of the functional $\frac{1}{2}\| \div \cdot \|_2^2$, we obtain
\begin{equation}
\label{RJ_FB3}
\frac{1}{2N_c} \| \div R_k^* R_k (\p - \bS_k (\q)) \|_2^2 \geq \frac{1}{2} \| \div R_k^* R_k \left( \p - \S (\q) \right) \|_2^2
 - \frac{1}{2} \left( 1 - \frac{1}{N_c} \right) \| \div R_k^* R_k (\p - \q ) \|_2^2 .
 \end{equation}
Summation of~\cref{RJ_FB3} over $k=1,\dots,N_c$ yields
\begin{equation}
\label{RJ_FB4}
\frac{1}{N_c} \sum_{k=1}^{N_c}\frac{1}{2} \| \div R_k^* R_k (\p - \bS_k (\q)) \|_2^2 \geq D(\p, \S (\q)) - \left( 1 - \frac{1}{N_c} \right) D(\p , \q).
\end{equation}
On the other hand, we have
\begin{equation} \begin{split}
\label{RJ_FB5}
F(\q) &- \frac{1}{N_c} \sum_{k=1}^{N_c} F_k (R_k \p ; \q) = \frac{1}{N_c} \sum_{k=1}^{N_c} \left[ F(\q) - F (R_k^* R_k \p + (I-R_k^* R_k) \q) \right] \\
&= \frac{1}{N_c} \sum_{k=1}^{N_c} \left[  - \left< \nabla F(\q), R_k^* R_k (\p - \q) \right> - \frac{1}{2} \| \div R_k^* R_k (\p -\q ) \|_2^2 \right] \\
&= - \frac{1}{N_c} \left< \nabla F (\q) , \p - \q \right> - \frac{1}{N_c} D(\p , \q) \\
&\geq - \frac{1}{N_c} (F(\p) -  F(\q)) - \frac{1}{N_c} D(\p, \q),
\end{split} \end{equation}
where the last inequality is due to the convexity of $F$.
Summation of~\cref{RJ_FB1,RJ_FB2,RJ_FB4,RJ_FB5} yields the desired result.
\end{proof}

Now, we are ready to prove $O(1/n)$ convergence of \cref{Alg:RJ}.
The proof presented in this paper uses the same argument as the proof of Theorem~3.1 in~\cite{BT:2009}.

% Theorem: Convergence rate of Relaxed Block Jacobi Method
\begin{theorem}
\label{Thm:RJ}
Let $\left\{ \p^{(n)} \right\}$ be the sequence generated by \cref{Alg:RJ} and $\p^*$ be a solution of~\cref{d_dual_ROF}.
Then, for any $n \geq 1$, we have
\begin{equation*}
F(\p^{(n)}) - F(\p^*) \leq \frac{N_c D(\p^* , \p^{(0)}) + (N_c - 1) \left( F(\p^{(0)}) - F(\p^*) \right) }{n}.
\end{equation*}
\end{theorem}
\begin{proof}
Using \cref{Lem:RJ_FB} with $\p = \p^*$ and $\q = \p^{(j)}$, we get
\begin{equation*}
\frac{1}{N_c} F(\p^*) + \left( 1 - \frac{1}{N_c} \right) F(\p^{(j)}) - F(\p^{(j+1)}) \geq D(\p^* , \p^{(j+1)}) - D(\p^* , \p^{(j)}).
\end{equation*}
Summation of this inequality over $j=0,1,\dots,n-1$ yields
\begin{equation}
\label{RJ1}
\frac{n}{N_c} F(\p^*) +\left( 1 - \frac{1}{N_c} \right) \left( F(\p^{(0)})- F(\p^{(n)}) \right) - \frac{1}{N_c} \sum_{j=1}^n F(\p^{(j)})
\geq D(\p^* , \p^{(n)}) - D(\p^* , \p^{(0)}).
\end{equation}
Also, by \cref{Lem:RJ_FB} with $\p = \q = \p^{(j)}$, we obtain
\begin{equation*}
F(\p^{(j)}) - F(\p^{(j+1)}) \geq D(\p^{(j)} , \p^{(j+1)}).
\end{equation*}
Multiplying this inequality by $j$ and summing over $j=1,\dots,n-1$ yields
\begin{equation}
\label{RJ2}
-n F(\p^{(n)}) + \sum_{j=1}^{n} F(\p^{(j)}) \geq \sum_{j=1}^{n-1} D(\p^{(j)}, \p^{(j+1)}).
\end{equation}
Adding~\cref{RJ1} times $N_c$ and~\cref{RJ2}, we obtain
\begin{equation*} \begin{split}
&n \left(F(\p^*) - F(\p^{(n)})\right) + (N_c - 1) \left( F(\p^{(0)}) - F(\p^{(n)}) \right) \\
&\hspace{0.5cm} \geq N_c\left( D(\p^* , \p^{(n)}) - D(\p^* , \p^{(0)}) \right) + \sum_{j=1}^{n-1} D(\p^{(j)}, \p^{(j+1)}).
\end{split} \end{equation*}
That is,
\begin{align*}
n \left(F(\p^*) - F(\p^{(n)})\right) &\geq (N_c - 1) \left( F(\p^{(n)}) - F(\p^{(0)}) \right) \\
&\quad+ N_c \left( D(\p^* , \p^{(n)}) - D(\p^* , \p^{(0)}) \right)
+ \sum_{j=1}^{n-1} D(\p^{(j)}, \p^{(j+1)}) \\
&\geq -(N_c - 1) \left( F(\p^{(0)}) - F(\p^*) \right) - N_c D(\p^* , \p^{(0)}).
\end{align*}
Therefore, we conclude that
\begin{equation*}
F(\p^{(n)}) - F(\p^*) \leq \frac{N_c D(\p^* , \p^{(0)}) + (N_c - 1) \left( F(\p^{(0)}) - F(\p^*) \right) }{n}.
\end{equation*}
\end{proof}

% Remark: Remark on CTWY
\begin{remark}
\label{Rem:CTWY}
In~\cite{CTWY:2015}, the relaxed block Jacobi method with overlapping domain decomposition has $O(1/n)$ convergence rate with the constant that tends to $\infty$ as the overlapping size tends to 0.
Therefore, the proof in~\cite{CTWY:2015} does not guarantee the convergence of the method in the nonoverlapping case.
\end{remark}

By the definition of~$D(\p^* , \p^{(0)})$ in~\cref{D_dist}, the estimate given in \cref{Thm:RJ} depends on the number of subdomains~$\N$. The following lemma describes such dependency in detail.

% Lemma: Dependency on \N
\begin{lemma}
\label{Lem:D}
Assume that $\Omega$ with~$M \times N$ pixels is partitioned into $\N = M_s \times N_s$ nonoverlapping rectangular subdomains of the same size.
For any $\p \in W$, we have
\begin{equation*}
\sum_{k=1}^{N_c} \| \div R_k^* R_k \p \|_2^2 \leq \| \div \p \|_2^2 + c_1 \| \p \|_{\infty}^2 ,
\end{equation*}
where $c_1$ is given by
\begin{equation}
\label{c1}
c_1 = 7\left( MN_s + M_s N - \frac{11}{7}M_s N_s \right).
\end{equation}
\end{lemma}
\begin{proof}
For $\p \in W$, we define $R^s \p = \p |_{\Omega_s}$ for~$s=1,\dots, \N$.
Then, we clearly have
\begin{equation*}
\sum_{k=1}^{N_c} \| \div R_k^* R_k \p \|_2^2 = \sum_{s=1}^{\N} \| \div (R^s)^* R^s \p \|_2^2 .
\end{equation*}
For $(i,j) \in \Omega$, by the definition of $\div$, one requires $\p_{ij}$, $\p_{i-1,j}$, and $\p_{i,j-1}$ to compute $(\div \p)_{ij}$.
Let $\tOmega_s$ and $\Omega_s^{\circ}$ be subsets of $\Omega$ defined by
\begin{equation*}
\tOmega_s = \bigcup_{(i,j) \in \Omega_s} \left\{ (i,j), (i+1,j), (i,j+1) \in \Omega \right\}
\end{equation*}
and
\begin{equation*}
\Omega_s^{\circ} = \left\{ (i,j) \in \Omega : (i,j), (i-1,j), (i,j-1) \in \Omega_s \right\},
\end{equation*}
respectively; see \cref{Fig:subdomain} for graphical description.
Since
\begin{equation*}
\p = (R^s)^* R^s \p \quad \textrm{in } \Omega_s,
\end{equation*}
one can observe that
\begin{equation*}
\supp \left( \div (R^s)^* R^s \p \right) \subset \tOmega_s
\end{equation*}
and
\begin{equation*}
\div (R^s)^* R^s \p = \div \p \quad \textrm{in } \Omega_s^{\circ}.
\end{equation*}
Thus, we obtain that
\begin{equation} \begin{split}
\label{edges}
&\quad\left\| \div (R^s)^* R^s \p \right\|_2^2 - \| (\div \p ) |_{\Omega_s} \|_2^2 \\
&\leq \sum_{(i,j) \in \tOmega_s \setminus \Omega_s} \left[ \left( \div (R^s)^* R^s \p \right)_{ij} \right]^2
+ \sum_{(i,j) \in \Omega_s \setminus \Omega_s^{\circ}} \left[ \left( \div (R^s)^* R^s \p \right)_{ij} \right]^2 \\
&= \sum_{(i,j) \in E_1} ( p_{ij}^1 + p_{ij}^2 - p_{i,j-1}^2 )^2 + \sum_{(i,j) \in E_2} ( - p_{i-1,j}^1 )^2 \\
&\quad+ \sum_{(i,j) \in E_3} ( p_{ij}^1 - p_{i-1,j}^1 + p_{ij}^2 )^2 + \sum_{(i,j) \in E_4} ( - p_{i,j-1}^2 )^2
+ \sum_{(i,j) \in E_5} ( p_{ij}^1 + p_{ij}^2 )^2,
\end{split} \end{equation}
where $E_1$, $E_2$, $E_3$, $E_4$, and $E_5$ are regions indicated in \cref{Fig:subdomain}(c).
Note that the cardinalities of these regions are $\frac{N}{N_s} - 1$, $\frac{N}{N_s}$, $\frac{M}{M_s} - 1$, $\frac{M}{M_s}$, and $1$, respectively.
By the Cauchy--Schwarz inequality, we have
\begin{align*}
\sum_{(i,j) \in E_1} ( p_{ij}^1 + p_{ij}^2 - p_{i,j-1}^2 )^2 &\leq \sum_{(i,j) \in E_1} 3\left[ (p_{ij}^1)^2 + (p_{ij}^2)^2 + (p_{i, j-1}^2)^2 \right] \\
&\leq \sum_{(i,j) \in E_1} 6 \| \p \|_{\infty}^2 = 6 \left( \frac{N}{N_s} - 1 \right) \| \p \|_{\infty}^2.
\end{align*}
Similarly, one obtains the followings:
\begin{eqnarray*}
\sum_{(i,j) \in E_2} ( - p_{i-1,j}^1 )^2 &\leq& \frac{N}{N_s} \| \p \|_{\infty}^2, \\
 \sum_{(i,j) \in E_3} ( p_{ij}^1 - p_{i-1,j}^1 + p_{ij}^2  )^2 &\leq& 6 \left( \frac{M}{M_s} - 1 \right) \| \p \|_{\infty}^2, \\
 \sum_{(i,j) \in E_4} ( - p_{i,j-1}^2 )^2 &\leq& \frac{M}{M_s} \| \p \|_{\infty}^2, \\
\sum_{(i,j) \in E_5} ( p_{ij}^1 + p_{ij}^2 )^2 &\leq& \| \p \|_{\infty}^2.
\end{eqnarray*}
Combining all inequalities stated above, we conclude that
\begin{align*}
\left\| \div (R^s)^* R^s \p \right\|_2^2 - \| (\div \p ) |_{\Omega_s} \|_2^2 &\leq \left[ 6\left( \frac{N}{N_s} - 1 \right) + \frac{N}{N_s} + 6\left( \frac{M}{M_s} - 1 \right) + \frac{M}{M_s} + 1 \right] \| \p \|_{\infty}^2\\
&= 7 \left( \frac{M}{M_s} + \frac{N}{N_s} - \frac{11}{7} \right) \| \p \|_{\infty}^2.
\end{align*}
Summing the last inequality for $s=1, \dots, \N$ yields
\begin{align*}
 \sum_{s=1}^{\N} \| \div (R^s)^* R^s \p \|_2^2 - \| \div \p \|_2^2 &= \sum_{s=1}^{\N} \left( \| \div (R^s)^* R^s \p \|_2^2 - \| (\div \p)|_{\Omega_s} \|_2^2\right) \\
 &\leq 7M_s N_s \left( \frac{M}{M_s} + \frac{N}{N_s} - \frac{11}{7}\right) \| \p \|_{\infty}^2 \\
 &= 7\left(MN_s + M_s N - \frac{11}{7}M_s N_s \right) \| \p \|_{\infty}^2 .
\end{align*}
This completes the proof.
\end{proof}

% Figure: Subdomain
\begin{figure}[]
\centering
\subfloat[][$\tOmega_s$]{ \includegraphics[height=4.6cm]{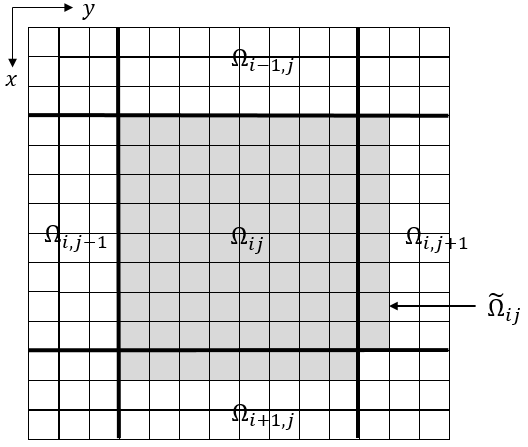} }
\quad
\subfloat[][$\Omega_s^{\circ}$] { \includegraphics[height=4.6cm]{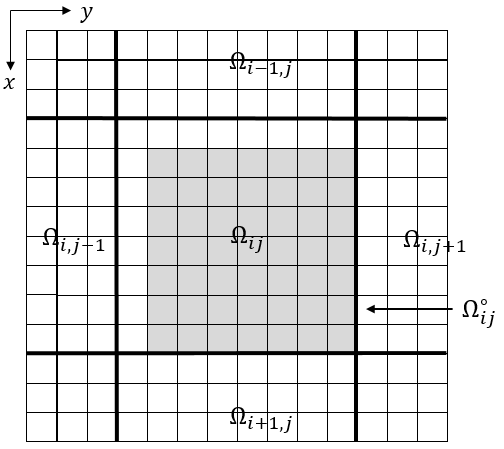} } \\
\subfloat[][$E_1 , E_2, E_3, E_4$, and $E_5$]{ \includegraphics[height=4.6cm]{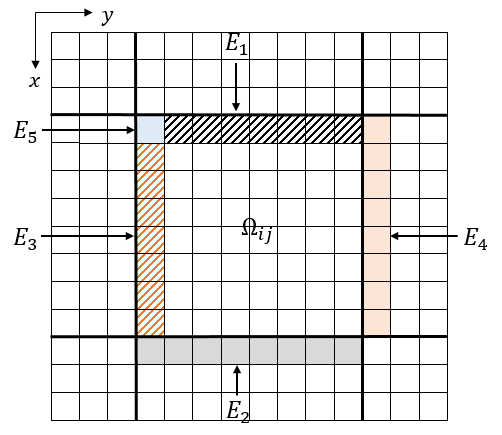} }
\quad\quad
\caption{Subsets of $\Omega$ introduced in the proof of \cref{Lem:D}}
\label{Fig:subdomain}
\end{figure}

% Remark: Case M_s = 1, N_s = 1
\begin{remark}
\label{Rem:stripe}
In the case of the stripe-shaped domain decomposition, say $M_s = 1$ and $N_s = \N$, one may obtain a sharper estimate than \cref{Lem:D} since we do not need to deal with the terms related to~$E_1$ and~$E_2$ in~\cref{edges}.
As a result, we obtain
\begin{equation*}
\sum_{k=1}^{N_c} \| \div R_k^* R_k \p \|_2^2 \leq \| \div \p \|_2^2 + \N (7M - 5) \| \p \|_{\infty}^2 .
\end{equation*}
\end{remark}

As a direct consequence of \cref{Lem:D}, we get the following upper bound for $D(\p^*, \p)$ for any $\p \in C$.

% Corollary: Estimate of D
\begin{corollary}
\label{Cor:D}
Let $\p^*$ be a solution of~\cref{d_dual_ROF}.
Then, for any $\p \in C$, we have
\begin{equation*}
D(\p^*, \p) \leq F(\p) - F(\p^*) + 2c_1,
\end{equation*}
where $c_1$ is given in~\cref{c1}.
\end{corollary}
\begin{proof}
From the optimality condition of~$\p^*$, we get
\begin{equation}
\label{CorD1}
\frac{1}{2} \| \div (\p - \p^*) \|_2^2 \leq F(\p) - F(\p^*).
\end{equation}
Also, since both $\p$ and $\p^*$ are in~$C$, we have
\begin{equation}
\label{CorD2}
\| \p - \p^* \|_{\infty} \leq 2.
\end{equation}
With \cref{Lem:D,CorD1,CorD2}, we readily obtain
\begin{align*}
D(\p^* , \p) &= \frac{1}{2} \sum_{k=1}^{N_c} \| \div R_k^* R_k (\p - \p^* ) \|_2^2 \\
&\leq \frac{1}{2} \| \div (\p - \p^*) \|_2^2 + \frac{c_1}{2} \| \p - \p^* \|_{\infty}^2 \\
&\leq F(\p) - F(\p^*) + 2c_1.
\end{align*}
\end{proof}

Combining \cref{Thm:RJ,Cor:D}, we obtain the following estimate of convergence rate, in which dependence on the size of the image and the number of subdomains is revealed.

% Corollary: Relaxed Jacobi method
\begin{corollary}
\label{Cor:RJ}
Let $\left\{ \p^{(n)} \right\}$ be the sequence generated by \cref{Alg:RJ} and $\p^*$ be a solution of~\cref{d_dual_ROF}.
Then, for any $n \geq 1$, we have
\begin{equation*}
F(\p^{(n)}) - F(\p^*) \leq \frac{N_c}{n} \left[ \left(2 - \frac{1}{N_c} \right) (F(\p^{(0)}) - F(\p^*)) + 2 c_1 \right],
\end{equation*}
where $c_1$ is given in~\cref{c1}.
\end{corollary}

% Remark: Convergence rate and the length of the subdomain interface
\begin{remark}
\label{Rem:interface}
One can obtain a relation between the convergence rate and the total length of the subdomain interfaces from \cref{Cor:RJ}.
Let $\Gamma_{st} = \partial \Omega_s \cap \partial \Omega_t$ for $s<t$ and $\Gamma = \bigcup_{s<t} \Gamma_{st}$.
If the decomposition $\{ \Omega_s \}_{s=1}^{\N}$ is of $M_s \times N_s$, it is clear that the length of $\Gamma$ is given by
\begin{equation*}
|\Gamma| = M(N_s - 1) + N(M_s - 1).
\end{equation*}
For fixed $M$ and $N$, $c_1$ satisfies
\begin{equation}
\label{interface}
c_1 \leq 7|\Gamma| + c_2,
\end{equation}
where $c_2$ is a constant independent of $M_s$ and $N_s$.
Therefore, it is expected that the convergence rate of \cref{Alg:RJ} may deteriorate as $|\Gamma|$ increases.
Using \cref{Rem:stripe}, the same bound as~\cref{interface} can be yielded for the stripe-shaped domain decomposition.
\end{remark}

% Subsection: Pre-relaxed Block Jacobi method
\subsection{Pre-relaxed Block Jacobi method}
First, we observe that \cref{Lem:RJ_FB} has the similar form  to one in the fundamental lemma of the forward-backward splitting~(see Lemma~2.3 of~\cite{BT:2009}):
\begin{equation}
\label{FB_fund}
F(\p) - F(\T (\q)) \geq \frac{L}{2} \left( \| \p - \T (\q) \|^2 - \| \p - \q \|^2 \right), \hspace{0.5cm} \p, \q \in C,
\end{equation}
where $\T$:~$W \rightarrow W$ is a certain operator which computes the next iterate of the algorithm, and~$L$ is a constant.
If we replace $F(\p)$ in~\cref{FB_fund} by $\frac{1}{N_c} F(\p) + \frac{N_c - 1}{N_c} F(\q)$, then we obtain the formula in \cref{Lem:RJ_FB}.
Note that an~$O(1/n)$ convergent algorithm with the forward-backward splitting~\cref{FB_fund} has the FISTA acceleration~\cite{BT:2009}.
However, even though Algorithm~1 has the~$O(1/n)$ convergence rate, it is difficult to apply the FISTA acceleration because of such differences.
Since the relaxation step is crucial to guarantee the convergence of the algorithm~\cite{CTWY:2015}, 
to apply the FISTA acceleration, we modify the algorithm so that a formula of the form~\cref{FB_fund} can be yielded while it still has the relaxation step.

We define an alternative local solution operator $\P_k$:~$W \rightarrow W_k$ by
\begin{equation}
\label{P_k}
\P_k (\p) = \argmin_{\p_k \in W_k} \left\{ F_k \left( N_c \p_k - (N_c - 1) R_k \p ; \p \right) + \chi_{C_k} (\p_k) \right\}, \hspace{0.5cm} \p \in W.
\end{equation}
The difference between two operators $\S_k$ and $\P_k$ is that, $\p_k$ in $\S_k$ is replaced by the relaxed term $N_c \p_k - (N_c - 1) R_k \p$ in $\P_k$.
Now, we propose a variant of \cref{Alg:RJ} which uses~$\P_k$ as the local solution operator; see \cref{Alg:PJ}.

% Algorithm (Pre-relaxed Block Jacobi Method)
\begin{algorithm}[]
\caption{Pre-relaxed Block Jacobi Method}
\begin{algorithmic}[]
\label{Alg:PJ}
\STATE Let $\p^{(0)} \in C$.
\FOR {$n=0,1,2,\dots$}
\STATE $\displaystyle \p^{(n+1)} = \sum_{k=1}^{N_c} R_k^* \P_k (\p^{(n)})$
\ENDFOR
\end{algorithmic}
\end{algorithm}

Compared to \cref{Alg:RJ}, the relaxation step is inside the local solution operator in \cref{Alg:PJ}.
That is, the relaxation step is done \textit{before} computing the local solution operator.
This is why we call \cref{Alg:PJ} the \textit{pre-relaxed} block Jacobi method.
The following lemma says that an application of the operator $\sum_{k=1}^{N_c} R_k^* \P_k$ is in fact a proximal descent step with respect to a pseudometric
\begin{equation*}
d(\p, \q) = \left( \sum_{k=1}^{N_c} \| \div R_k^* R_k (\p - \q) \|_2^2 \right)^{\frac{1}{2}} = \sqrt{2D(\p, \q)}, \quad \p, \q \in W.
\end{equation*}

% Lemma: Pre-relaxed block Jacobi method is ISTA.
\begin{lemma}
\label{Lem:PJ_ISTA}
For $\q \in C$, $\sum_{k=1}^{N_c} R_k^* \P_k (\q)$ is a solution of the minimization problem
\begin{equation}
\label{PJ_ISTA}
\min_{\p \in W} \left\{ F(\q) + \left< \nabla F(\q) , \p - \q \right> + N_c D(\p, \q) + \chi_C (\p) \right\} .
\end{equation}
\end{lemma}
\begin{proof}
For $\p_k \in W_k$, let
\begin{equation*} \begin{split}
G_k (\p_k) &= F_k \left( N_c \p_k - (N_c - 1) R_k \q ; \q \right) \\
&= \frac{1}{2} \left\| \div \left( R_k^* (N_c\p_k - (N_c - 1) R_k \q) + (I-R_k^* R_k )\q \right) + \alpha f \right\|_2^2.
\end{split} \end{equation*}
Let us define
\begin{equation}
\label{A_k}
A_k = R_k \div^* \div R_k^*, \hspace{0.5cm} A = \sum_{k=1}^{N_c} R_k^* A_k R_k.
\end{equation}
Then, we have
\begin{align*}
\nabla G_k(\p_k) &= N_c R_k \div^* \left[ \div \left( R_k^* (N_c \p_k - (N_c - 1) R_k \q ) + (I-R_k^* R_k )\q \right) + \alpha f \right] \\
&=  N_c [ R_k \nabla F(\q) + N_c A_k (\p_k - R_k \q)].
\end{align*}
Thus, the optimality condition of~\cref{P_k} reads as
\begin{equation}
\label{PJ_ISTA1}
\left< R_k \nabla F(\q) + N_c A_k (\P_k (\q) - R_k \q) , \p_k - \P_k (\q) \right> \geq 0, \hspace{0.5cm} \p_k \in C_k.
\end{equation}
Assembly of~\cref{PJ_ISTA1} for~$k=1,\dots, N_c$ gives
\begin{equation}
\label{P_opt}
\left< \nabla F (\q) + N_c A \left( \sum_{k=1}^{N_c} R_k^* \P_k (\q) - \q \right) , \p - \sum_{k=1}^{N_c} R_k^* \P_k (\q) \right> \geq 0, \hspace{0.5cm} \p \in C, 
\end{equation}
which is the optimality condition of~\cref{PJ_ISTA}.
Therefore,  $\sum_{k=1}^{N_c} R_k^* \P_k (\q)$ is a solution of~\cref{PJ_ISTA}.
\end{proof}

Together with \cref{Lem:Lip,Lem:PJ_ISTA}, we obtain a formula for \cref{Alg:PJ} which has the exactly same form as~\cref{FB_fund}.
The fundamental lemma for \cref{Alg:PJ} is presented in the following.

% Lemma: Forward-Backward structure of the pre-relaxed Jacobi method
\begin{lemma}
\label{Lem:PJ_FB}
For any $\p, \q \in C$, we have
\begin{equation*}
\frac{1}{N_c}F(\p) - \frac{1}{N_c} F\left(\sum_{k=1}^{N_c} R_k^* \P_k (\q) \right) \geq D\left( \p , \sum_{k=1}^{N_c} R_k^* \P_k (\q) \right) - D(\p , \q).
\end{equation*}
\end{lemma}
\begin{proof}
For the sake of convenience, we define $A$ as in~\cref{A_k} and
\begin{equation*}
\P (\q) = \sum_{k=1}^{N_c} R_k^* \P_k (\q).
\end{equation*}
Invoking \cref{Lem:Lip} yields
\begin{equation}
\label{PJ_FB1}
F (\p) - F(\P (\q)) \geq F(\p) - F(\q) - \left< \nabla F(\q), \P (\q)  - \q \right> - N_c D(\P (\q) , \q). 
\end{equation}
By \cref{Lem:PJ_ISTA}, $\P (\q)$ satisfies~\cref{P_opt}, that is,
\begin{equation}
\label{PJ_FB2}
\left< \nabla F(\q) + N_c A (\P (\q) - \q) , \p - \P (\q) \right> \geq 0. 
\end{equation}
Also, by the convexity of $F$, we have
\begin{equation}
\label{PJ_FB3}
F(\p) \geq F(\q) + \left< \nabla F(\q) , \p - \q \right>.
\end{equation}
Then, we get the desired result by summation of~\cref{PJ_FB1,PJ_FB2,PJ_FB3}.
\end{proof}

Since \cref{Lem:PJ_FB} has exactly the same form as Lemma~2.3 of~\cite{BT:2009}, analysis of the~$O(1/n)$ convergence rate is straightforward.

% Theorem: Convergence rate of Pre-relaxed Block Jacobi Method
\begin{theorem}
\label{Thm:PJ}
Let $\left\{ \p^{(n)} \right\}$ be the sequence generated by \cref{Alg:PJ} and $\p^*$ be a solution of~\cref{d_dual_ROF}.
Then, for any $n \geq 1$, we have
\begin{equation*}
F(\p^{(n)}) - F(\p^*) \leq \frac{N_c }{n} \left[ F(\p^{(0)}) - F(\p^*) + 2c_1 \right],
\end{equation*}
where $c_1$ is given in~\cref{c1}.
\end{theorem}
\begin{proof}
In Theorem~3.1 of~\cite{BT:2009}, we replace $L(f)$ and $\| \cdot \|$ by $N_c$ and
\begin{equation*}
\| \cdot \| = \left( \sum_{k=1}^{N_c}\| \div R_k^* R_k (\cdot) \|_2^2 \right)^{\frac{1}{2}},
\end{equation*}
respectively, the notations we use.
Consequently, we have the following estimate:
\begin{equation}
\label{PJ_rate}
F(\p^{(n)}) - F(\p^*) \leq \frac{N_c D(\p^* , \p^{(0)})}{n}.
\end{equation}
Application of \cref{Cor:D} to~\cref{PJ_rate} yields the desired result.
\end{proof}

% Remark: Constant on the estimate
\begin{remark}
\label{Rem:const}
In \cref{Thm:PJ}, we gave a little sharper bound than \cref{Thm:RJ}.
Indeed, our numerical experiments presented in \cref{Sec:Numerical} will show that \cref{Thm:PJ} has a faster energy decay than \cref{Thm:RJ} at initial dozens of iterations.
\end{remark}

In addition, an accelerated version of \cref{Alg:PJ} with the $O(1/n^2)$ convergence rate can be designed by the same way as FISTA~\cite{BT:2009}.
It is summarized in \cref{Alg:FPJ}.

% Algorithm (Fast Pre-relaxed Block Jacobi Method)
\begin{algorithm}[]
\caption{Fast Pre-relaxed Block Jacobi Method}
\begin{algorithmic}[]
\label{Alg:FPJ}
\STATE Let $\p^{(0)} = \q^{(0)} \in C$ and $t_0 = 1$.
\FOR {$n=0,1,2,\dots$}
\STATE $\displaystyle \p^{(n+1)} = \sum_{k=1}^{N_c} R_k^* \P_k \q^{(n)}$
\STATE $\displaystyle t_{n+1} = \frac{1 + \sqrt{1 + 4t_n^2}}{2}$
\STATE $\displaystyle \q^{(n+1)} = \p^{(n+1)} + \frac{t_n - 1}{t_{n+1}} (\p^{(n+1)} - \p^{(n)})$
\ENDFOR
\end{algorithmic}
\end{algorithm}

Finally, we state the convergence theorem for \cref{Alg:FPJ}.
\Cref{Thm:FPJ} says that \cref{Alg:FPJ} is superior to \cref{Alg:RJ,Alg:PJ}.
In \cref{Sec:Numerical}, we will present some numerical examples where \cref{Alg:FPJ} shows better performance than others.

% Theorem: Convergence rate of Fast Pre-relaxed Block Jacobi Method
\begin{theorem}
\label{Thm:FPJ}
Let $\left\{ \p^{(n)} \right\}$ be the sequence generated by \cref{Alg:FPJ} and $\p^*$ be a solution of~\cref{d_dual_ROF}.
Then, for any $n \geq 1$, we have
\begin{equation*}
F(\p^{(n)}) - F(\p^*) \leq \frac{4N_c }{(n+1)^2} \left[ F(\p^{(0)}) - F(\p^*) + 2c_1 \right],
\end{equation*}
where $c_1$ is given in~\cref{c1}.
\end{theorem}
\begin{proof}
With the same argument as in Theorem~4.4 of~\cite{BT:2009}, we obtain
\begin{equation}
\label{FPJ_rate}
F(\p^{(n)}) - F(\p^*) \leq \frac{4N_c D(\p^* , \p^{(0)}) }{(n+1)^2}.
\end{equation}
Then, application of \cref{Cor:D} to~\cref{FPJ_rate} completes the proof.
\end{proof}

% Remark: Convergence analyses are independent of alpha, f
\begin{remark}
\label{Rem:f_alpha}
In \cref{Thm:FPJ}, the $O(1/n^2)$ convergence is satisfied regardless of $f$ and $\alpha$ since they only affect on $F$.
That is, neither the noise level of a given image nor the value of the weight parameter $\alpha$ interferes with the $O(1/n^2)$ convergence of \cref{Alg:FPJ}.
Similar arguments apply to \cref{Alg:RJ,Alg:PJ}.
\end{remark}

% Section: Local problems
\section{Local problems}
\label{Sec:Local}

% Figure: Domain decomposition
\begin{figure}[]
\centering
\subfloat[][Window-shaped DD,\\ \centering$\N = 4\times4$]{ \includegraphics[height=4.5cm]{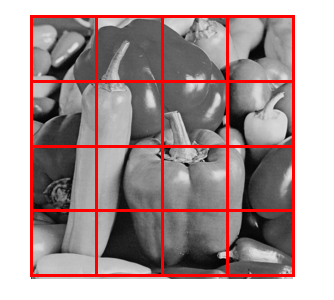} }
\quad
\subfloat[][Stripe-shaped DD,\\ \centering$\N = 6$] { \includegraphics[height=4.5cm]{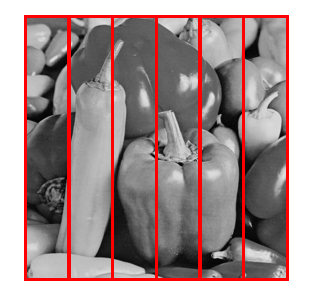} }
\caption{Two different shapes of domain decomposition}
\label{Fig:DD}
\end{figure}

In this section, we discuss how to deal with local problems of the block Jacobi methods introduced in this paper.
The coloring technique used in the proposed methods is also considered.

First, we assume that the domain decomposition $\left\{ \Omega_s \right\}_{s=1}^{\N}$ is of the window shape as shown in \cref{Fig:DD}(a).
All subdomains $\Omega_s$ have the same size.
Let $\Omega_{ij}$ be the subdomain on the~$i$-th row and the~$j$-th column.
We define the local function space $W^s$ as the collection of functions from~$\Omega_s$ to~$\mathbb{R}^2$.
Also, we define the restriction operator $R^s$:~$W \rightarrow W^s$ similarly to~\cref{R_k} and~$C^s = R^s C$.

Local problems in $\Omega_s$ have the following general form:
\begin{equation}
\label{local}
\min_{\p_s \in W^s} \frac{1}{2} \| \div (R^s)^* \p_s + g \|_2^2 + \chi_{C^s} (\p_s)
\end{equation}
for some $g \in V$.
Let $\tOmega_s$ be the subset of $\Omega$ containing $\Omega_s$ and its adjacent lines of pixels on the bottom and the right sides; see \cref{Fig:subdomain}(a).
Also, we define~$\tilde{V}^s$ as the collection of functions from $\tOmega_s$ to~$\mathbb{R}$.
One can observe that $\div (R^s)^* \p_s \in \tilde{V}_s$ for all $\p_s \in W^s$.
Thus, it is natural to define the local divergence operator $\div_s$:~$W^s \rightarrow \tilde{V}^s$ by
\begin{equation}
\label{div_s}
\div_s \p_s = \div (R^s)^* \p_s, \hspace{0.5cm} \p_s \in W^s,
\end{equation}
and the local gradient operator $\nabla_s$:~$\tilde{V}^s \rightarrow W^s$ by $\nabla_s = - \div_s^*$.
With these local operators, it is straightforward to obtain the equivalent primal and primal-dual form of~\cref{local} as follows~(see, for example,~\cite{Rockafellar:2015} for details):
\begin{subequations}
\begin{equation}
\label{local_primal}
\min_{\tilde{u}_s \in \tilde{V}^s} \left\{ \frac{1}{2} \| \tilde{u}_s - g \|_2^2 + \| \nabla_s \tilde{u}_s \|_1 \right\},
\end{equation}
\begin{equation}
\label{local_pd}
\min_{\tilde{u}_s \in \tilde{V}^s} \max_{\p_s \in W^s} \left\{ \left< \nabla_s \tilde{u}_s , \p_s \right> + \frac{1}{2} \| \tilde{u}_s - g \|_2^2 - \chi_{C^s} (\p_s) \right\}. 
\end{equation}
\end{subequations}
Consequently, existing state-of-the-art solvers for the ROF model can be adopted for local problems of the proposed methods.
For instance, one may solve~\cref{local} by FISTA~\cite{BT:2009},~\cref{local_primal} by the augmented Lagrangian method~\cite{WT:2010}, or~\cref{local_pd} by the first order primal-dual algorithm~\cite{CP:2011}.

% Figure: Dependency
\begin{figure}[]
\centering
\includegraphics[height=4.6cm]{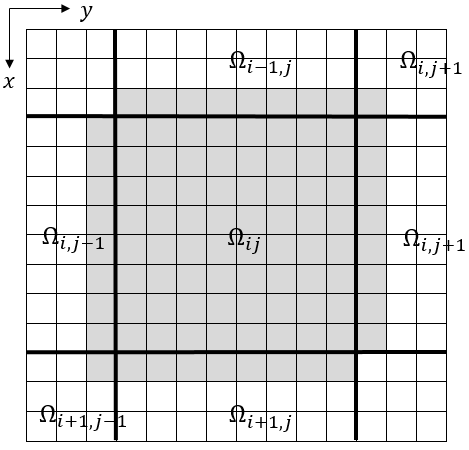}
\caption{Local problem~\cref{local} is dependent on the degrees of freedom in the marked area.}
\label{Fig:dependency}
\end{figure}

% Figure: Coloring technique
\begin{figure}[]
\centering
\includegraphics[height=4.2cm]{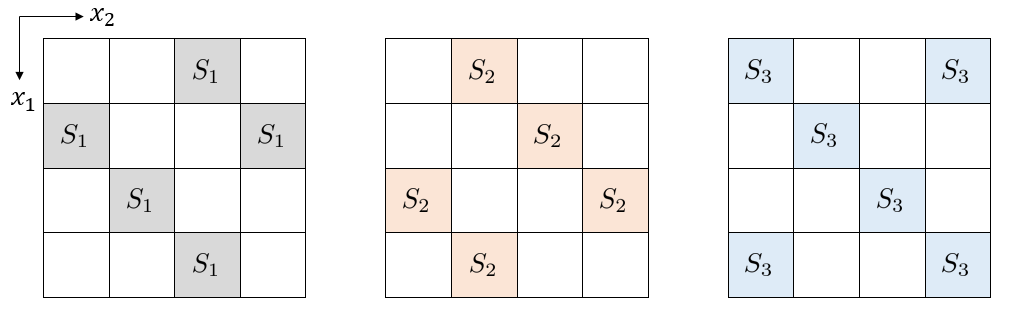}
\caption{Domain decomposition with the coloring technique, $N_c = 3$}
\label{Fig:coloring}
\end{figure}

Next, we consider the coloring technique used in the proposed methods.
As we mentioned above, we assign the same color to subdomains if local problems associated to them are solved independently.
To compute~$g \in \tilde{V}^s$ in~\cref{local}, we require the values of~$\q \in W$ in the area marked in \cref{Fig:dependency}, i.e.,~$\tOmega_s$ plus its adjacent lines of pixels on the top and the left sides.
It means that the local problem~\cref{local} on~$\Omega_{ij}$ is dependent on~$\Omega_{i-1,j}$, $\Omega_{i-1,j+1}$, $\Omega_{i,j-1}$, $\Omega_{i,j+1}$, $\Omega_{i+1,j-1}$, and~$\Omega_{i+1, j}$.
Only~$\Omega_{i-1, j-1}$ and~$\Omega_{i+1,j+1}$ can have the same color as~$\Omega_{ij}$ among its neighboring subdomains.
It can be accomplished with 3 colors; each subdomain~$\Omega_{ij}$ is colored with the color~$(i-j) \textrm{ mod }3$.
See \cref{Fig:coloring} for visual description.
We notice that the same coloring technique was introduced in~\cite{LN:2017} for primal DDMs.

We conclude the section with a remark on domain decomposition shapes.
Since DDMs are designed for use on distributed memory computers, it is important to reduce the amount of communications among processors.
In this sense, the window-shaped domain decomposition shown in \cref{Fig:DD}(a) is most efficient because the total length of the subdomain interfaces is minimized.
Moreover, we observed in \cref{interface} that the convergence rate of block Jacobi methods rely on the length of the subdomain interfaces.
However, in the case when we have only a few number of processors, the stripe-shaped domain decomposition shown in \cref{Fig:DD}(b) can be a good alternative.
Since the stripe-shaped domain decomposition can be colored with only~2 colors, one may expect better convergence rate; recall that all the convergence theorems presented in this paper are dependent on the number of colors~$N_c$.
Numerical comparison between two types of domain decomposition will be presented in the next section.
 
% Section: Numerical experiments
\section{Numerical experiments}
\label{Sec:Numerical}

% Figure: Test images
\begin{figure}[]
\centering
\subfloat[][Peppers $512 \times 512$]{ \includegraphics[height=4.5cm]{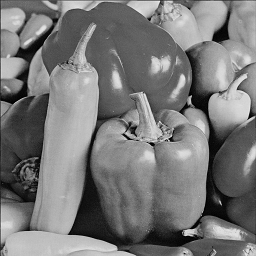} }
\quad
\subfloat[][Boat $2048 \times 3072$] { \includegraphics[height=4.5cm]{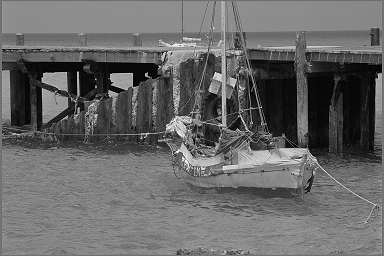} } \\
\subfloat[][Corrupted Peppers \\ \centering (PSNR: 19.11)]{ \includegraphics[height=4.5cm]{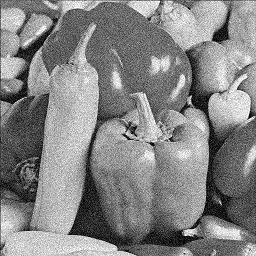} }
\quad
\subfloat[][Corrupted Boat (PSNR: 19.09)] { \includegraphics[height=4.5cm]{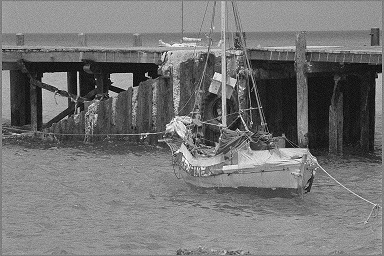} }
\caption{Test images for the numerical experiments}
\label{Fig:test}
\end{figure}

In this section, we present numerical results for the block Jacobi methods introduced in this paper.
All codes were programmed using ANSI C with OpenMPI and compiled by Intel Parallel Studio XE. All computations were performed on a computer cluster composed of seven machines, where each machine is equipped with two Intel Xeon SP-6148 CPUs~(2.4GHz, 20C), 192GB RAM, and the operating system CentOS~7.4 64bit.
Two test images ``Peppers~$512\times 512$'' and ``Boat~$2048 \times 3072$'' are used in our numerical experiments.
Both are corrupted with additive Gaussian noise with mean~$0$ and variance~$0.05$; see \cref{Fig:test}.
As a measurement of the quality of denoising, we use the peak-signal-to-noise ratio~(PSNR) defined by
\begin{equation*}
\mathrm{PSNR} = 10 \log_{10} \left( \frac{\mathrm{MAX}^2 \cdot |\Omega|}{ \| u - f_{\textrm{orig}} \|_2^2 }\right),
\end{equation*}
where $\mathrm{MAX}=1$ is the possible maximum value of pixel intensity and $f_{\textrm{orig}}$ is the original clean image.
The weight parameter~$\alpha$ in~\cref{d_dual_ROF} is chosen as~5, 10, and 20 in our experiments.

For all experiments, we set the initial guess as~$\p^{(0)} = \mathbf{0}$.
To reduce the time elapsed in solving local problems, the local solutions~$\p_s^{(n-1)}$ from the previous iteration are chosen as initial guesses for the local problems at each iteration.

% Figure: Energy decay of the block Jacobi methods
\begin{figure}[]
\centering
\subfloat[][Peppers $512 \times 512$, $\alpha = 5$]{ \includegraphics[height=4.7cm]{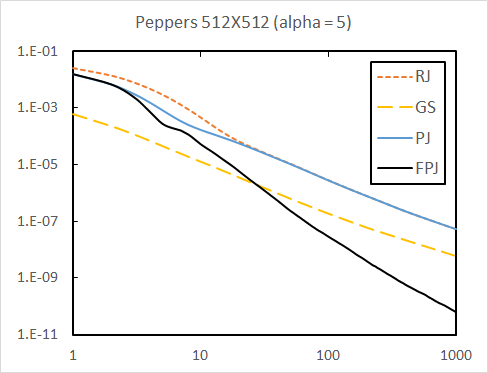} }
\quad\quad
\subfloat[][Boat $2048 \times 3072$, $\alpha = 5$] { \includegraphics[height=4.7cm]{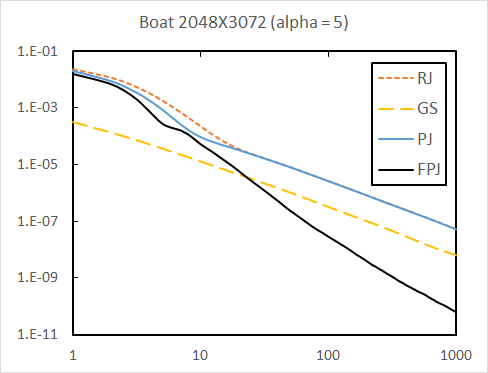} } \\

\subfloat[][Peppers $512 \times 512$, $\alpha = 10$]{ \includegraphics[height=4.7cm]{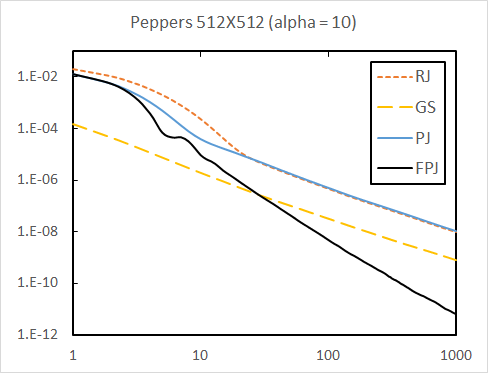} }
\quad\quad
\subfloat[][Boat $2048 \times 3072$, $\alpha = 10$] { \includegraphics[height=4.7cm]{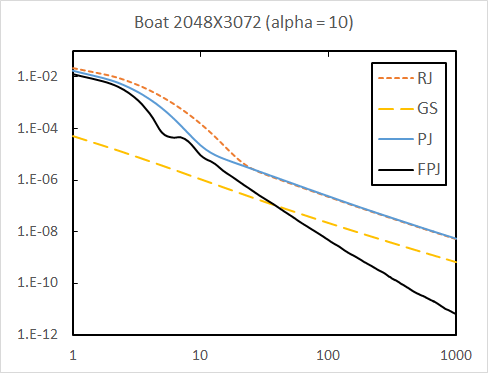} } \\

\subfloat[][Peppers $512 \times 512$, $\alpha = 20$]{ \includegraphics[height=4.7cm]{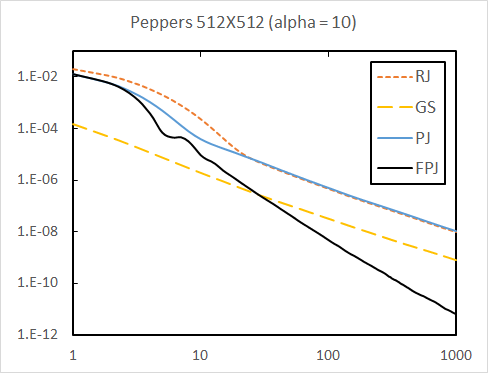} }
\quad\quad
\subfloat[][Boat $2048 \times 3072$, $\alpha = 20$] { \includegraphics[height=4.7cm]{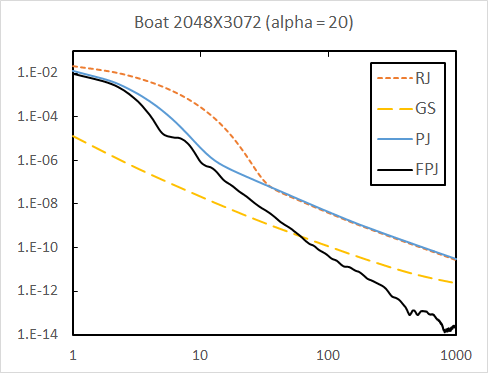} }
\caption{Decay of $\frac{F(\p^{(n)}) - F(\p^*)}{F(\p^*)}$ in several block methods~($\N = 8 \times 8$)}
\label{Fig:energy_J}
\end{figure}

% Figure: Results of the block Jacobi methods
\begin{figure}[]
\centering
\subfloat[][RJ (PSNR: 24.55)]{ \includegraphics[width=4.5cm]{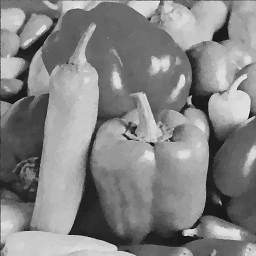} }
\quad
\subfloat[][PJ (PSNR: 24.55)] { \includegraphics[width=4.5cm]{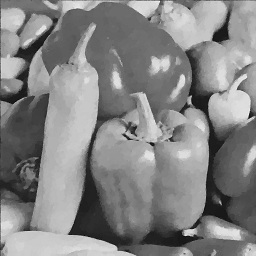} }
\quad
\subfloat[][FPJ (PSNR: 24.55)] { \includegraphics[width=4.5cm]{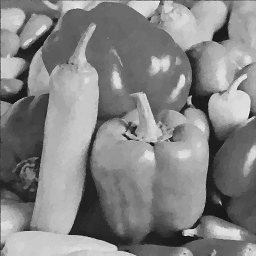} }
\\
\subfloat[][RJ (PSNR: 24.86)]{ \includegraphics[width=4.5cm]{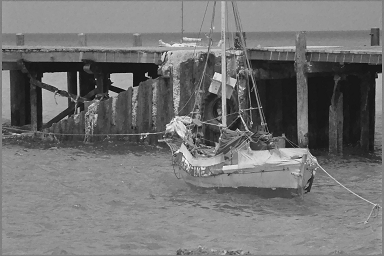} }
\quad
\subfloat[][PJ (PSNR: 24.86)] { \includegraphics[width=4.5cm]{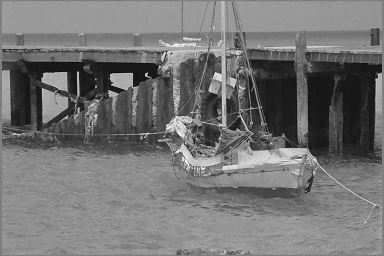} }
\quad
\subfloat[][FPJ (PSNR: 24.86)] { \includegraphics[width=4.5cm]{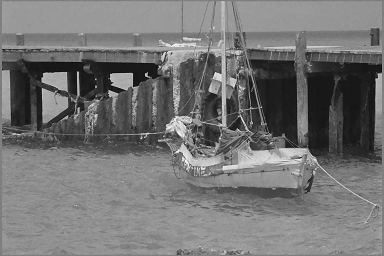} }
\caption{Results of the block Jacobi methods~($\N = 8 \times 8$, $\alpha = 10$)}
\label{Fig:J}
\end{figure}

First, we compare the energy decay of several block methods: relaxed block Jacobi~(\cref{Alg:RJ}, RJ), block Gauss--Seidel~(Algorithm~1 in~\cite{HL:2015} , GS), pre-relaxed block Jacobi~(\cref{Alg:PJ}, PJ), and fast pre-relaxed block Jacobi~(\cref{Alg:FPJ}, FPJ) methods.
\Cref{Fig:energy_J} plots $\frac{F(\p^{(n)}) - F(\p^*)}{F(\p^*)}$ of the block methods per iteration in $\log$-$\log$ scale, where $\p^*$ is a solution of the full-dimension problem~\cref{d_dual_ROF} obtained by $10^6$ FISTA~\cite{BT:2009} iterations.
The number of subdomains~$\N$ is fixed as~$8 \times 8$.
Local problems were solved by FISTA with the stop criterion
\begin{equation*}
\frac{\| \div_s \p_s^{(n+1)} - \div_s \p_s^{(n)} \|_2}{\| \div_s \p_s^{(n+1)} \|_2 } < 10^{-9} \quad\textrm{or}\quad
n = 50,
\end{equation*}
where the operator $\div_s$ was defined in~\cref{div_s}.
The above criterion was chosen so that stagnations of the energy do not occur~(see Figure~5 in~\cite{LPP:2019}).
As we justified theoretically in \cref{Sec:Jacobi}, the energy functional of FPJ converges to the minimum much faster than other methods.
As noted in \Cref{Rem:f_alpha}, such fast convergence occurs regardless of the given image $f$ and the weight parameter $\alpha$.
In particular, the convergence rate of FPJ is faster than that of GS, while other block Jacobi methods converge more slowly than GS.
We notice that the $O(1/n)$ convergence of GS can be deduced as a direct consequence of~\cite{CP:2015,ST:2016}.
It is observed that the energy decay of FPJ is not monotone.
This is due to the nature of FISTA, and one may use MFISTA~\cite{BT:2009b} instead of FISTA acceleration to ensure monotone decay of the energy.
Note that the monotonicity of MFISTA given in Appendix~A of~\cite{BT:2009b} can be applied to our framework directly.
Interestingly, as we mentioned in \cref{Rem:const}, PJ seems to be faster than RJ at first 20 iterations, but eventually shows similar convergence results compared to RJ in our numerical experiments.
Meanwhile, as shown in \cref{Fig:J}, all the results from three block Jacobi methods with $\alpha = 10$ have the same PSNR and are not visually distinguishable.
The same applies to the cases $\alpha = 5, 20$, and we omit the results for those cases.

% Figure: Results of the block Jacobi methods w.r.t. alpha
\begin{figure}[]
\centering
\subfloat[][$\alpha=5$ (PSNR: 23.91)]{ \includegraphics[width=4.5cm]{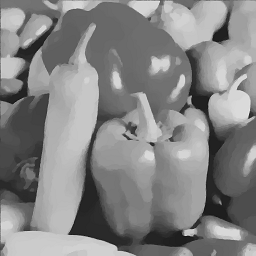} }
\quad
\subfloat[][$\alpha=10$ (PSNR: 24.55)] { \includegraphics[width=4.5cm]{./peppers_FPJ_10.png} }
\quad
\subfloat[][$\alpha=20$ (PSNR: 23.88)] { \includegraphics[width=4.5cm]{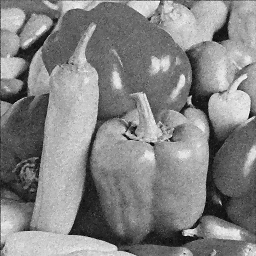} }
\\
\subfloat[][$\alpha=5$ (PSNR: 24.12)]{ \includegraphics[width=4.5cm]{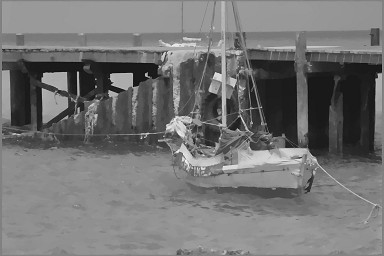} }
\quad
\subfloat[][$\alpha=10$ (PSNR: 24.86)] { \includegraphics[width=4.5cm]{./boat_FPJ_10.png} }
\quad
\subfloat[][$\alpha=20$ (PSNR: 24.27)] { \includegraphics[width=4.5cm]{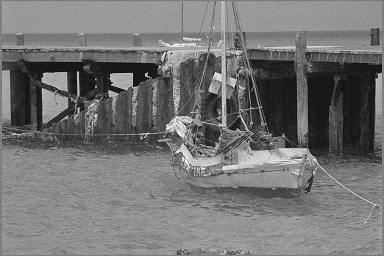} }
\caption{Results of the fast pre-relaxed block Jacobi method with various values of $\alpha$ ($\N = 8 \times 8$)}
\label{Fig:FPJ_alpha}
\end{figure}

% Table: Performance of the block Jacobi method w.r.t. alpha
\begin{table}
\centering
\begin{tabular}{| c | c | c c c |} \hline
Test image & $\alpha$  & PSNR& iter & \begin{tabular}{c}wall-clock\\time (sec) \end{tabular}\\
\hline
\multirow{3}{*}{\shortstack{Peppers \\ $512 \times 512$}}
& 5 & 23.88& 26 & 0.34 \\ 
& 10 & 24.55 & 11 & 0.15 \\
& 20 & 23.91 & 8 & 0.12 \\
\hline
\multirow{3}{*}{\shortstack{Boat \\ $2048 \times 3072$}}
& 5 & 24.27 & 94 & 46.23 \\ 
& 10 & 24.86 & 10 & 7.08 \\
& 20 & 24.12 & 6 & 4.64\\
\hline
\end{tabular}
\caption{Performance of the fast pre-relaxed block Jacobi method with various values of $\alpha$ ($\N = 8 \times 8$)}
\label{Table:FPJ_alpha}
\end{table}

For practical use in image restoration, the stop criteria need not be too strict; see~\cite{CTWY:2015} for details.
In the following experiments, we use
\begin{equation}
\label{stop_cr}
\frac{F(\p^{(n)}) - F(\p^*)}{F(\p^*)} < 10^{-5}
\end{equation}
for outer iterations, and
\begin{equation*}
\frac{\| \div_s \p_s^{(n+1)} - \div_s \p_s^{(n)} \|_2}{\| \div_s \p_s^{(n+1)} \|_2 } < 10^{-4} \quad\textrm{or}\quad
n = 50,
\end{equation*}
for local problems.

We compare the results of FPJ with respect to the weight parameter~$\alpha$.
\Cref{Fig:FPJ_alpha} shows the results of FPJ with $\alpha=5, 10,$ and~$20$ with $\N = 8 \times 8$.
As it is well-known for the ROF model, the staircase effect becomes strong as $\alpha$ decreases.
Among three results, $\alpha = 10$ gives the highest PSNR.
We observe from \cref{Table:FPJ_alpha} that both the number of FPJ iterations and the wall-clock time increase as $\alpha$ decreases.

% Table: Performance for the block Jacobi method w.r.t. \N
\begin{table}
\centering
\begin{tabular}{| c | c | c c c |} \hline
Test image & $\N$ & PSNR & iter & \begin{tabular}{c}wall-clock\\time (sec)\end{tabular} \\
\hline
\multirow{4}{*}{\shortstack{\begin{phantom}1\end{phantom} \\ Peppers \\ $512 \times 512$}}
& 1 & 24.55 & - & 0.71 \\ \cline{2-5}
& $2 \times 2$ & 24.55 & 9 &  0.70 \\
& $4 \times 4$ & 24.55 & 10 & 0.23 \\
& $8 \times 8$ & 24.55 & 11 &  0.15 \\
& $16 \times 16$ & 24.55 & 14 & 0.20 \\
\hline
\multirow{4}{*}{\shortstack{\begin{phantom}1\end{phantom} \\ Boat \\ $2048 \times 3072$}}
& 1 & 24.86 & - & 53.24 \\ \cline{2-5}
& $2 \times 2$ & 24.86 & 10 & 50.14 \\
& $4 \times 4$ & 24.86 & 10 & 14.47 \\
& $8 \times 8$ & 24.86 & 10 & 7.08 \\
& $16 \times 16$ & 24.86 & 10 & 4.70 \\
\hline
\end{tabular}
\caption{Performance of the fast pre-relaxed block Jacobi method with various numbers of subdomains~$\N$~($\alpha = 10$)}
\label{Table:FPJ_N}
\end{table}

% Figure: 16 X 16 subdomains results
\begin{figure}[]
\centering
\subfloat[][PSNR: 24.55]{ \includegraphics[height=4.5cm]{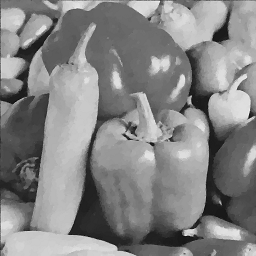} }
\quad
\subfloat[][PSNR: 24.86] { \includegraphics[height=4.5cm]{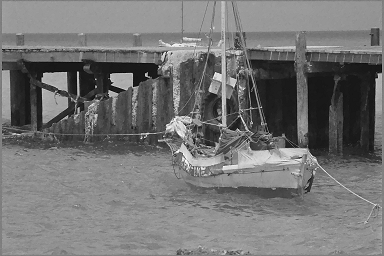} } 
\caption{Results of the fast pre-relaxed block Jacobi method~($\N = 16\times 16$, $\alpha = 10$)}
\label{Fig:FPJ_N}
\end{figure} 

\Cref{Table:FPJ_N} shows the performance of FPJ with various numbers of subdomains~$\N$ when $\alpha = 10$.
The case~$\N = 1$ presents the results of FISTA applied to the full-dimension problem for~\cref{d_dual_ROF} using the stopping criterion~\cref{stop_cr}.
In the case of the test image ``Peppers $512 \times 512$,'' the wall-clock time of FPJ decreases as $\N$ grows from 1 to $8 \times 8$,  but increases when $\N$ becomes $16 \times 16$.
It is because the size of ``Peppers $512 \times 512$'' is so small that the portion of communication time between processors in the wall-clock time is not negligible compared to computation time if we use $\N = 16 \times 16$ subdomains.
For a sufficiently large test image ``Boat $2048 \times 3072$,'' the wall-clock time of FPJ decreases as~$\N$ increases.
It shows the efficiency of FPJ as a parallel solver for large-scale images.
We also observe numerically that the number of FPJ iterations increases as $\N$ grows.
\Cref{Fig:FPJ_N} shows the denoised test images by FPJ with~$\N = 16 \times 16$.
As they show no trace of the subdomain interfaces and have the same PSNR as the case~$\N = 1$, we can say that FPJ correctly solves~\cref{d_dual_ROF} and recovers images.

% Figure: Window DD vs Stripe DD
\begin{figure}[]
\centering
\subfloat[][Peppers, 4 subdomains]{ \includegraphics[height=4.7cm]{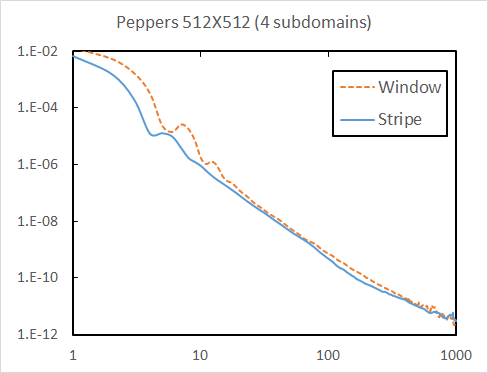} }
\quad\quad
\subfloat[][Boat, 4 subdomains] { \includegraphics[height=4.7cm]{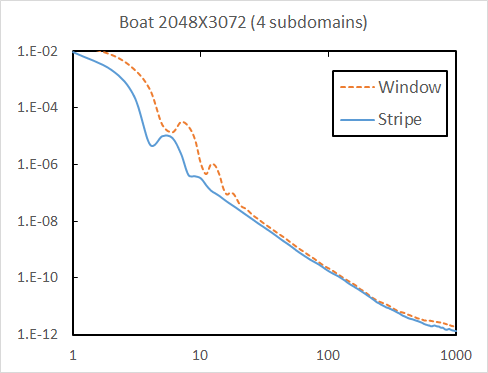} } \\
\subfloat[][Peppers, 256 subdomains]{ \includegraphics[height=4.7cm]{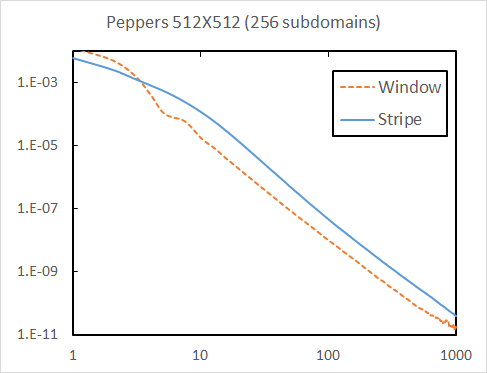} }
\quad\quad
\subfloat[][Boat, 256 subdomains] { \includegraphics[height=4.7cm]{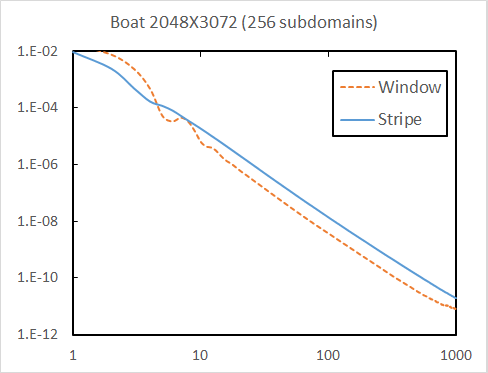} }
\caption{Decay of $\frac{F(\p^{(n)}) - F(\p^*)}{F(\p^*)}$ in the fast pre-relaxed block Jacobi method with window/stripe-shaped domain decompositions~($\alpha = 10$)}
\label{Fig:energy_DD}
\end{figure}

% Table: Performance for the block Jacobi method w.r.t. DD shape
\begin{table}
\centering
\begin{tabular}{| c | c | c c c |} \hline
Test image & $\N$ & PSNR & iter & \begin{tabular}{c}wall-clock\\time (sec)\end{tabular} \\
\hline
\multirow{4}{*}{\shortstack{\begin{phantom}1\end{phantom} \\ Peppers \\ $512 \times 512$}}
& 1 & 24.55 & - & 0.71 \\ \cline{2-5}
& 4 & 24.55 & 7 &  0.57 \\
& 16 & 24.55 & 8 & 0.22 \\
& 64 & 24.55 & 13 &  0.19 \\
& 256 & 24.55 & 23 & 1.05 \\
\hline
\multirow{4}{*}{\shortstack{\begin{phantom}1\end{phantom} \\ Boat \\ $2048 \times 3072$}}
& 1 & 24.86 & - & 53.24 \\ \cline{2-5}
& 4 & 24.86 & 7 & 37.54 \\
&16 & 24.86 & 7 & 10.53 \\
&64 & 24.86 & 9 & 6.86 \\
& 256 & 24.86 & 14 & 7.01 \\
\hline
\end{tabular}
\caption{Performance of the fast pre-relaxed block Jacobi method with various numbers of subdomains~$\N$, the stripe-shaped case~($\alpha = 10$)}
\label{Table:FPJ_stripe}
\end{table}

Finally, we compare the convergence rate of FPJ for the window-shaped domain decomposition and the stripe-shaped one.
\Cref{Fig:energy_DD} shows the decay of $\frac{F(\p^{(n)}) - F(\p^*)}{F(\p^*)}$ for both domain decompositions when $\N = 4$ and~$256$.
Since $N_c = 2$ for the stripe shape and $N_c = 3$ for the window shape, the stripe-shaped domain decomposition converges faster than the window-shaped one when $\N = 4$.
However, since the total lengths of the subdomain interfaces are $O(\N^{1/2})$ and $O(\N )$ for the window and stripe shape, respectively, as we observed in \cref{Rem:interface}, the window-shaped domain decomposition is faster than the stripe-shaped one when $\N$ is large enough; see \cref{Fig:energy_DD}(c) and~(d).
\Cref{Table:FPJ_stripe} provides the performance of FPJ with the stripe-shaped domain decomposition with various values of $\N$ when $\alpha = 10$.
Comparing \cref{Table:FPJ_N,Table:FPJ_stripe}, it can be inferred that if $\N$ is small, the stripe-shaped domain decomposition is more efficient than the window-shaped one, and if $\N$ is large enough, the opposite holds.

% Section: Conclusion
\section{Conclusion}
\label{Sec:Conclusion}
We proved the~$O(1/n)$ convergence rate of the nonoverlapping relaxed block Jacobi method for the dual ROF model proposed in~\cite{HL:2015}.
Then, incorporating the FISTA acceleration technique~\cite{BT:2009} into this method, we proposed the~$O(1/n^2)$ convergent fast pre-relaxed block Jacobi method.
We implemented various block decomposition methods and showed that the fast pre-relaxed block Jacobi method outperforms others numerically.
The fast pre-relaxed block Jacobi method also showed good performance as a parallel solver for large-scale images.

Since our convergence analysis cannot be applied directly to overlapping cases, we plan to investigate ways to extend our acceleration methodology to overlapping methods~\cite{CTWY:2015} in our future work.
In addition, to design scalable DDMs, construction of the coarse grid correction will be considered.

\bibliographystyle{siamplain}
\bibliography{bibliography}
\end{document}